\documentclass[leqno,12pt]{article} 
\usepackage{amsmath, amssymb}
\usepackage{amsthm} 
\newcommand\supp{\mathop{\rm supp}}
\newcommand\esssup{\mathop{\rm ess \, sup}}
\newcommand\essinf{\mathop{\rm ess \, inf}}

\theoremstyle{plain} 
\newtheorem{theorem}{\indent\sc Theorem}[section]
\newtheorem{lemma}[theorem]{\indent\sc Lemma}

\newtheorem{proposition}[theorem]{\indent\sc Proposition}

\theoremstyle{definition} 
\newtheorem{definition}[theorem]{\indent\sc Definition}
\newtheorem{remark}[theorem]{\indent\sc Remark}

\makeatletter
\def\address#1#2{\begingroup
\noindent\parbox[t]{7.8cm}{%
\small{\scshape\ignorespaces#1}\par\vskip1ex
\noindent\small{\itshape E-mail address}%
\/: #2\par\vskip4ex}\hfill%
\endgroup}%
\makeatother
\title{\uppercase{A molecular reconstruction theorem for} $H^{p(\cdot)}_{\omega}(\mathbb{R}^{n})$} 
\author{
\textsc{Pablo Rocha} 
}
\date{} 
\begin{document}

\maketitle

\footnote{ 
2020 \textit{Mathematics Subject Classification}.
Primary 42B30; Secondary 42B20, 42B25, 42B35.
}
\footnote{ 
\textit{Key words and phrases}:
weighted variable Hardy spaces, molecular decomposition, singular integrals, Riesz potential.
}

\begin{abstract}
In this article we give a molecular reconstruction theorem for $H_{\omega}^{p(\cdot)}(\mathbb{R}^{n})$. As an application of this result 
and the atomic decomposition developed in \cite{Ho1} we show that classical singular integrals can be extended to bounded operators on
$H_{\omega}^{p(\cdot)}(\mathbb{R}^{n})$. We also prove, for certain exponents $q(\cdot)$ and certain weights $\omega$, that 
Riesz potential $I_{\alpha}$, with $0 < \alpha < n$, can be extended to a bounded operator from
$H^{p(\cdot)}_{\omega}(\mathbb{R}^{n})$ into $H^{q(\cdot)}_{\omega}(\mathbb{R}^{n})$, for 
$\frac{1}{p(\cdot)} := \frac{1}{q(\cdot)} + \frac{\alpha}{n}$.
\end{abstract}

\section{Introduction}

In the celebrated paper \cite{Fefferman}, Fefferman and Stein defined the Hardy space $H^{p}(\mathbb{R}^{n})$. Since then, the study of Hardy spaces has received the attention of a substantial number of researchers. One of the most remarkable results for the study of Hardy
spaces is the atomic characterization of $H^{p}(\mathbb{R}^{n})$ (see \cite{Coifman, Latter}). Roughly speaking, every distribution 
$f \in H^{p}$ can be expressed of the form
\[
f = \sum_j \lambda_j a_j, 
\]
where the $a_j$'s are $p$ - atoms, $\{ \lambda_j \} \in \ell^{p}$ and $\| f \|^{p}_{H^{p}} \approx \sum_j | \lambda_j |^{p}$.
For $0 < p \leq 1$, an $p$ - atom is a function $a(\cdot)$ supported on a cube $Q$ such that
\[
\| a \|_{\infty} \leq |Q|^{-1/p} \,\,\  \text{and} \,\, \int x^{\alpha} a(x) dx = 0 \,\,\, \text{for all} \,\, |\alpha| \leq 
n \left( \frac{1}{p}-1 \right).
\]
Such decomposition allows to study the behavior of certain operators on $H^{p}$ by focusing one's attention on individual atoms.
In principle, the continuity of an operator $T$ on $H^{p}$ can often be proved by estimating $Ta$ when $a(\cdot)$ is an atom. In \cite{Coifman2} was observed that, in general, the atoms are not mapped into atoms. However, for many convolution operators (e.g.: singular integrals or Riesz potentials), $m=Ta$ behaves like an atom. These functions $m$ were called \textit{molecules}, their properties as well as the molecular characterization of $H^{p}(\mathbb{R}^{n})$ were established in \cite{Taibleson}. Then, in essence, the continuity
$H^{p} \to H^{p}$ of an operator reduces to showing that it maps atoms into molecules.

The theory of variable Hardy spaces $H^{p(\cdot)}(\mathbb{R}^{n})$ was first developed by Nakai and Sawano in \cite{Nakai}, where they proved
the infinite atomic decomposition and molecular decomposition for $H^{p(\cdot)}$. As a corollary of these decompositions, they obtained the boundedness of singular integral operators on $H^{p(\cdot)}$.
Later, Cruz-Uribe and Wang in \cite{Uribe3} independently considered the same problem under a weaker condition for variable exponents 
$p(\cdot)$. They gave a finite atomic decomposition for $H^{p(\cdot)}$ and also proved the boundedness of singular integral operators on
$H^{p(\cdot)}$.
Both theories prove equivalent definitions in terms of maximal operators using different approaches. In \cite{Rocha1}, the author jointly with Urciuolo proved  the $H^{p(\cdot)} - L^{q(\cdot)}$ boundedness of certain generalized Riesz potentials and the $H^{p(\cdot)} - H^{q(\cdot)}$ boundedness of Riesz potential via the infinite atomic and molecular decomposition developed in \cite{Nakai}. In \cite{Rocha3}, the author gave another proof of the results obtained in \cite{Rocha1}, but by using the finite atomic decomposition given in \cite{Uribe3}.

Recently, Kwok-Pun Ho in \cite{Ho1} developed the weighted theory for variable Hardy spaces on $\mathbb{R}^{n}$. He established the atomic decompositions for the weighted Hardy spaces with variable exponents $H^{p(\cdot)}_{\omega}(\mathbb{R}^{n})$ and also revealed some intrinsic structures of atomic decomposition for Hardy type spaces. His results generalize the infinite atomic decomposition obtained in \cite{Nakai}. By means of the atomic decomposition given in \cite{Ho1}, the author in \cite{Rocha4} proved, for certain exponents 
$q(\cdot)$ and certain weights $\omega$, the $H^{p(\cdot)}_{\omega} - L^{q(\cdot)}_{\omega}$ boundedness for Riesz potential. 

The purpose of this article is to generalize the molecular decomposition obtained in \cite[Theorem 5.2]{Nakai} to weighted variable Hardy spaces by using the framework developed in \cite{Ho1}.
As an application of our molecular decomposition, we prove two theorems concerning the $H^{p(\cdot)}_{\omega} - H^{p(\cdot)}_{\omega}$ boundedness of a class of singular integral operators and the $H^{p(\cdot)}_{\omega} - H^{q(\cdot)}_{\omega}$ boundedness of Riesz potential. These results are established in Theorems \ref{Boundedness T} and \ref{Boundedness Ialpha} below.

This paper is organized as follows. Section 2 gives weighted variable Hardy spaces $H^{p(\cdot)}_{\omega}$ and some of their preliminary results. Section 3 contains a molecular reconstruction theorem for $H^{p(\cdot)}_{\omega}$. The boundedness of singular integrals on 
$H^{p(\cdot)}_{\omega}$ is established in Section 4. The $H^{p(\cdot)}_{\omega} - H^{q(\cdot)}_{\omega}$ boundedness of 
Riesz potential is proved in Section 5. 

\

{\bf Notation.} The symbol $A\lesssim B$ stands for the inequality $A \leq c B$ for some positive constant $c$. The symbol $A \approx B$ 
stands for $B \lesssim A \lesssim B$. We denote by $Q\left( z,r\right)$ the cube centered at $z \in \mathbb{R}^{n}$ with side lenght $r$. Given a cube $Q = Q(z, r)$, we set $kQ = Q(z,kr)$ and $\ell(Q) = r$. For a measurable subset $E \subseteq \mathbb{R}^{n}$ we denote 
by $|E|$ and $\chi_{E}$ the Lebesgue measure of $E$ and the characteristic function of $E$ respectively. Given a real number $s \geq 0$, 
we write $\lfloor s \rfloor$ for the integer part of $s$. As usual we denote with $\mathcal{S}(\mathbb{R}^{n})$ the space of smooth and rapidly decreasing functions and with $\mathcal{S}'(\mathbb{R}^{n})$ the dual space. If $\beta$ is the multiindex 
$\beta =(\beta_{1},...,\beta _{n})$, then $|\beta| =\beta _{1}+...+\beta _{n}.$

\section{Preliminaries}

We begin with the definition of weighted variable Lebesgue spaces $L^{p(\cdot)}_{\omega}(\mathbb{R}^{n})$.

\

Let $p(\cdot) : \mathbb{R}^{n} \to (0, \infty)$ be a measurable function. Given a measurable set $E$, let
\[
p_{-}(E) = \essinf_{ x \in E } p(x), \,\,\,\, \text{and} \,\,\,\, p_{+}(E) = \esssup_{x \in E} p(x).
\]
When $E = \mathbb{R}^{n}$, we will simply write $p_{-} := p_{-}(\mathbb{R}^{n})$ and $p_{+} := p_{+}(\mathbb{R}^{n})$.

Given a measurable function $f$ on $\mathbb{R}^{n}$, define the modular $\rho$ associated with $p(\cdot)$ by
\[
\rho(f) = \int_{\mathbb{R}^{n}} |f(x)|^{p(x)} dx.
\]
We define the variable Lebesgue space $L^{p(\cdot)} = L^{p(\cdot)}(\mathbb{R}^{n})$ to be the set of all measurable functions $f$ such that, for some $\lambda > 0$, $\rho \left( f/\lambda \right) < \infty$. This becomes a quasi normed space when equipped with the Luxemburg norm
\[
\| f \|_{L^{p(\cdot)}} = \inf \left\{ \lambda > 0 : \rho \left( f/\lambda \right) \leq 1 \right\}.
\]

Given a weight $\omega$, i.e.: a locally integrable function on $\mathbb{R}^{n}$ such that $0 < \omega(x) < \infty$ almost everywhere, we define the weighted variable Lebesgue space $L^{p(\cdot)}_{\omega}$ as the set of all measurable functions 
$f : \mathbb{R}^{n} \to \mathbb{C}$ such that $\| f \omega \|_{L^{p(\cdot)}} < \infty$. If $f \in L^{p(\cdot)}_{\omega}$, we define 
its "norm" by
\[
\| f \|_{L^{p(\cdot)}_\omega} := \| f \omega \|_{L^{p(\cdot)}}. 
\]

The following result follows from the definition of the $L^{p(\cdot)}_{\omega}$-norm.

\begin{lemma} \label{potencia s}
Given a measurable function $p(\cdot) : \mathbb{R}^{n} \to (0, \infty)$ with $0 < p_{-} \leq p_{+} < \infty$ and a weight 
$\omega$, then for every $s > 0$,
\[
\| f \|_{L^{p(\cdot)}_\omega}^{s} = \| |f|^{s} \|_{L^{p(\cdot)/s}_{\omega^{s}}}.
\]
\end{lemma}

Next, we introduce the weights used in \cite{Ho1} to define weighted Hardy spaces with variable exponents.

\begin{definition} (See \cite[Remark 2.5]{Rocha4})\label{pesos Wp}
Let $p(\cdot) : \mathbb{R}^{n} \to (0, \infty)$ be a measurable function with $0 < p_{-} \leq p_{+} < \infty$. We define
$\mathcal{W}_{p(\cdot)}$ as the set of all weights $\omega$ such that
\[
(i) \,\, \text{there exists} \,\, 0 < p_{\ast} < \min \{1, p_{-} \} \,\, \text{such that} \,\, 
\| \chi_Q \|_{L^{p(\cdot)/p_{\ast}}_{\omega^{p_{\ast}}}} < \infty, \,\,\, \text{and}
\]
\[ 
\| \chi_Q \|_{L^{(p(\cdot)/p_{\ast})'}_{\omega^{-p_{\ast}}}} < \infty, \,\,\, \text{for all cube} \,\, Q;
\]
\[
(ii) \,\, \text{there exist $\kappa > 1$ and $s > 1$ such that Hardy-Littlewood maximal}
\]
\[ 
\text{operator $M$ is bounded on $L^{(sp(\cdot))'/\kappa}_{\omega^{-\kappa/s}}$}. 
\]
\end{definition}

For a measurable function $p(\cdot) : \mathbb{R}^{n} \to (0, \infty)$ such that $0 < p_{-}\leq p_{+} < \infty$ and 
$\omega \in \mathcal{W}_{p(\cdot)}$, in \cite{Ho1} the author give a variety of distinct approaches, based on
differing definitions, all lead to the same notion of weighted variable Hardy space $H^{p(\cdot)}_{\omega}$. 

We recall some terminologies and notations from the study of maximal functions. Given $N \in \mathbb{N}$, define 
\[
\mathcal{F}_{N}=\left\{ \varphi \in \mathcal{S}(\mathbb{R}^{n}):\sum\limits_{\left\vert \mathbf{\beta }\right\vert \leq N}\sup\limits_{x\in \mathbb{R}%
^{n}}\left( 1+\left\vert x\right\vert \right) ^{N}\left\vert \partial^{%
\mathbf{\beta }}\varphi (x)\right\vert \leq 1\right\}.
\]
For any $f \in \mathcal{S}'(\mathbb{R}^{n})$, the grand maximal function of $f$ is given by 
\[
\mathcal{M}_N f(x)=\sup\limits_{t>0}\sup\limits_{\varphi \in \mathcal{F}_{N}}\left\vert \left( \varphi_t \ast f\right)(x) \right\vert,
\]
where $\varphi_t(x) = t^{-n} \varphi(t^{-1} x)$. It is common to denote the grand maximal by $\mathcal{M}$.

\begin{definition} Let $p(\cdot):\mathbb{R}^{n} \to ( 0,\infty)$, $0 < p_{-} \leq p_{+} < \infty $, and 
$\omega \in \mathcal{W}_{p(\cdot)}$. The weighted variable Hardy space $H^{p(.)}_{\omega}(\mathbb{R}^{n})$ is the set of 
all $f \in \mathcal{S}'(\mathbb{R}^{n})$ for which $\| \mathcal{M}f \|_{L^{p(\cdot)}_{\omega}} < \infty$. In this case we define 
$\| f \|_{H^{p(\cdot)}_{\omega}} := \| \mathcal{M}f \|_{L^{p(\cdot)}_{\omega}}$.
\end{definition}

\begin{definition} \label{def atom} Let $p(\cdot):\mathbb{R}^{n} \to ( 0,\infty)$, $0 < p_{-} \leq p_{+} < \infty $, $p_{0} > 1$, and 
$\omega \in \mathcal{W}_{p(\cdot)}$. Fix an integer $N \geq 1$. A function $a(\cdot)$ 
on $\mathbb{R}^{n}$ is called a $\omega-(p(\cdot), p_{0}, N)$ atom if there exists a cube $Q$ such that

$a_{1})$ $\supp ( a ) \subset Q$,

$a_{2})$ $\| a \|_{L^{p_{0}}}\leq \frac{| Q |^{\frac{1}{p_{0}}}}{\| \chi _{Q} \|_{L^{p(\cdot)}_{\omega}}}$,

$a_{3})$ $\displaystyle{\int} x^{\beta}  a(x) \, dx = 0$ for all $| \beta | \leq N$.
\end{definition}

Now, we introduce two indices, which are related to the intrinsic structure of the atomic decomposition of $H^{p(\cdot)}_{\omega}$.
Given $\omega \in \mathcal{W}_{p(\cdot)}$, we write
\[
s_{\omega, \, p(\cdot)} := \inf \left\{ s \geq 1 : M \,\, \text{is bounded on} \,\, L^{(sp(\cdot))'}_{\omega^{-1/s}} \right\}
\]
and
\[
\mathbb{S}_{\omega, \, p(\cdot)} := \left\{ s \geq 1 : M \,\, \text{is bounded on} \,\, L^{(sp(\cdot))'/\kappa}_{\omega^{-\kappa/s}} \,\,
\text{for some} \,\, \kappa > 1 \right\}.
\]
Then, for every fixed $s \in \mathbb{S}_{\omega, \, p(\cdot)}$, we define
\[
\kappa_{\omega, \, p(\cdot)}^{s} := \sup \left\{ \kappa > 1 : M \,\, \text{is bounded on} \,\, L^{(sp(\cdot))'/\kappa}_{\omega^{-\kappa/s}} \right\}.
\]
The index $\kappa_{\omega, \, p(\cdot)}^{s}$ is used to measure the left-openness of the boundedness of $M$ on the family 
$\left\{ L^{(sp(\cdot))'/\kappa}_{\omega^{-\kappa/s}} \right\}_{\kappa > 1}$. The index $s_{\omega, \, p(\cdot)}$ is related to the vanishing moment condition and the index $\kappa_{\omega, \, p(\cdot)}^{s}$ is related to the size condition of the atoms 
(see \cite[Theorems 5.3 and 6.3]{Ho1}). These indices are also related to the structure of our molecular decomposition 
(see Definition \ref{def molecule} and Theorem \ref{molecular reconst} below).

\begin{definition} Let $p(\cdot):\mathbb{R}^{n} \to ( 0,\infty)$ be a measurable function such that $0 < p_{-} \leq p_{+} < \infty$, 
and let $\omega \in \mathcal{W}_{p(\cdot)}$. Given a sequence of scalars $\{ \lambda_j \}_{j=1}^{\infty}$, a family of 
cubes $\{ Q_j \}_{j=1}^{\infty}$ and $0 < \theta < \infty$, we define
\[
\mathcal{A} \left( \{ \lambda_j \}_{j=1}^{\infty}, \{ Q_j \}_{j=1}^{\infty}, p(\cdot), \omega, \theta \right) := 
\left\| \sum_{j=1}^{\infty} \left( \frac{|\lambda_j |}{\| \chi_{Q_j} \|_{L^{p(\cdot)}_{\omega}}} \right)^{\theta} \chi_{Q_j} \right\|_{L^{p(\cdot)/\theta}_{\omega^{\theta}}}^{1/\theta}.
\]
\end{definition}

The following theorem is a version of the atomic decomposition for $H^{p(\cdot)}_{\omega}$ obtained in \cite{Ho1}.

\begin{theorem} \label{w atomic decomp}
Let $1 < p_0 < \infty$, $p(\cdot) : \mathbb{R}^{n} \to (0, \infty)$ be a measurable function with $0 < p_{-} \leq p_{+} < \infty$ and 
$\omega \in \mathcal{W}_{p(\cdot)}$. Then, for every $f \in H^{p(\cdot)}_{\omega}(\mathbb{R}^{n}) \cap L^{p_0}(\mathbb{R}^{n})$ and
every integer $N \geq \lfloor n s_{\omega, \, p(\cdot)} - n \rfloor$ fixed, there exist a sequence of scalars 
$\{ \lambda_j \}_{j=1}^{\infty}$, a sequence of cubes $\{ Q_j \}_{j=1}^{\infty}$ and $\omega - (p(\cdot), p_0, N)$ atoms $a_j$ supported 
on $Q_j$ such that $f = \displaystyle{\sum_{j=1}^{\infty} \lambda_j a_j}$ converges in $L^{p_0}(\mathbb{R}^{n})$ and
\begin{equation} \label{Hpw norm atomic}
\mathcal{A} \left( \{ \lambda_j \}_{j=1}^{\infty}, \{ Q_j \}_{j=1}^{\infty}, p(\cdot), \omega, \theta \right) 
\lesssim \| f \|_{H^{p(\cdot)}_{\omega}}, \,\,\, \text{for all} \,\,\, 0 < \theta < \infty,
\end{equation}
where the implicit constant in (\ref{Hpw norm atomic}) is independent of $\{ \lambda_j \}_{j=1}^{\infty}$, $\{ Q_j \}_{j=1}^{\infty}$, and 
$f$.
\end{theorem}

\begin{proof}
The existence of a such atomic decomposition is guaranteed by \cite[Theorem 6.2]{Ho1}. Its construction is analogous to that given for classical Hardy spaces (see \cite[Chapter III]{Stein}). So, following the proof in \cite[Theorem 3.1]{Rocha2}, we obtain the convergence of the atomic series to $f$ in $L^{p_0}(\mathbb{R}^{n})$.
\end{proof}

The following three results will be useful in what follows.

\begin{lemma} (\cite[Lemma 3.4]{Rocha4}) \label{2Q}
Let $p(\cdot) : \mathbb{R}^{n} \to (0, \infty)$ be a measurable function with $0 < p_{-} \leq p_{+} < \infty$. If 
$\omega \in \mathcal{W}_{p(\cdot)}$, then, for every cube $Q \subset \mathbb{R}^{n}$,
\[
\| \chi_{2\sqrt{n}Q} \|_{L^{p(\cdot)}_\omega}  \approx \| \chi_{Q} \|_{L^{p(\cdot)}_\omega} .
\]
\end{lemma}

We say that an exponent function $p(\cdot) : \mathbb{R}^{n} \to (0, \infty)$ such that $0 < p_{-} \leq p_{+} < \infty$ belongs 
to $\mathcal{P}^{\log}(\mathbb{R}^{n})$, if there exist two positive constants $C$ and $C_{\infty}$ such that $p(\cdot)$ satisfies the 
local log-H\"older continuity condition, i.e.:
\[
|p(x) - p(y)| \leq \frac{C}{-\log(|x-y|)}, \,\,\, |x-y| \leq \frac{1}{2},
\]
and is log-H\"older continuous at infinity, i.e.:
\[
|p(x) - p_{\infty}| \leq \frac{C_{\infty}}{\log(e+|x|)}, \,\,\, x \in \mathbb{R}^{n},
\]
for some $p_{-} \leq p_{\infty} \leq p_{+}$.

We define the set $\mathcal{S}_{0}(\mathbb{R}^{n})$ by
\[
\mathcal{S}_{0}(\mathbb{R}^{n}) =\left\{ \varphi \in \mathcal{S}(\mathbb{R}^{n}) : \int x^{\beta} \varphi(x) dx = 0, \,\, \text{for all} \, 
\beta \in \mathbb{N}_{0}^{n} \right\}.
\] 

\begin{proposition} (\cite[Proposition 2.1]{Ho2}) \label{dense}
Let $p(\cdot) \in \mathcal{P}^{\log}(\mathbb{R}^{n})$ with $0 < p_{-} \leq p_{+} < \infty$. If $\omega \in \mathcal{W}_{p(\cdot)}$, then
$\mathcal{S}_{0}(\mathbb{R}^{n}) \subset H^{p(\cdot)}_{\omega}(\mathbb{R}^{n})$ densely.
\end{proposition}

We conclude this preliminaries with a supporting result, which will allow us to study the behavior of Riesz potential 
on $H^{p(\cdot)}_{\omega}$.

\begin{proposition} (\cite[Corollary 4.2]{Rocha4}) \label{Aq menor Ap}
Let $0 < \alpha < n$, $q(\cdot) : \mathbb{R}^{n} \to (0, \infty)$ be a measurable function with 
$0 < q_{-} \leq q_{+} < \infty$ and $\omega \in \mathcal{W}_{q(\cdot)}$. If $\frac{1}{p(\cdot)} := \frac{1}{q(\cdot)} + \frac{\alpha}{n}$  and $\| \chi_Q \|_{L^{q(\cdot)}_{\omega}} \approx |Q|^{-\alpha/n} \| \chi_Q \|_{L^{p(\cdot)}_{\omega}}$ for every cube $Q$, 
then for any sequence of scalars $\{ \lambda_j \}_{j=1}^{\infty}$, any family of cubes $\{ Q_j \}_{j=1}^{\infty}$, and any
$\theta \in (0, \infty)$ fixed we have
\[
\mathcal{A} \left( \{ \lambda_j \}_{j=1}^{\infty}, \{ Q_j \}_{j=1}^{\infty}, q(\cdot), \omega, \theta \right)
\lesssim
\mathcal{A} \left( \{ \lambda_j \}_{j=1}^{\infty}, \{ Q_j \}_{j=1}^{\infty}, p(\cdot), \omega, \theta \right).
\]
\end{proposition}

\section{Molecular reconstruction for $H^{p(\cdot)}_{\omega}(\mathbb{R}^{n})$}

Our definition of $\omega$ - molecule is a slight modification from that given in \cite{Nakai} by Nakai and Sawano for variable Hardy spaces.
Namely, we replace the amount $\| \chi _{Q} \|_{L^{p(\cdot)}}$ by $\| \chi _{Q} \|_{L^{p(\cdot)}_{\omega}}$.

\begin{definition} \label{def molecule}
Let $p(\cdot):\mathbb{R}^{n} \to ( 0,\infty)$ be a measurable function with $0 < p_{-} \leq p_{+} < \infty $ and 
let $\omega \in \mathcal{W}_{p(\cdot)}$. Let $p_{0} > 1$ and $d_{p(\cdot)} := \lfloor n s_{\omega, p(\cdot)} -n \rfloor$. 
A function $m(\cdot)$ on $\mathbb{R}^{n}$ is called a 
$\omega-(p(\cdot), p_{0}, d_{p(\cdot)})$ molecule centered at a cube $Q = Q(z,r)$ (with side length $\ell(Q)=r$) if 

$m_{1})$ $\| m \|_{L^{p_{0}}(2\sqrt{n}Q)}\leq \displaystyle{\frac{| Q |^{\frac{1}{p_{0}}}}{\| \chi _{Q} \|_{L^{p(\cdot)}_{\omega}}}}$,

$m_{2})$ $|m(x)| \leq \| \chi _{Q} \|_{L^{p(\cdot)}_{\omega}}^{-1} 
\left( \displaystyle{ 1 + \frac{|x-z|}{\ell(Q)} } \right)^{-2n-2d_{p(\cdot)}-3}$ for all 
$x \in \mathbb{R}^{n} \setminus Q(z,2\sqrt{n}r)$,

$m_{3})$ $\displaystyle{\int} x^{\beta}  m(x) \, dx = 0$ for all multi-index $\beta$ with $| \beta | \leq d_{p(\cdot)}$.
\end{definition}

\begin{remark} \label{molecule Lp}
The conditions $m_1)$ and $m_2)$ imply that $\| m \|_{L^{p_0}(\mathbb{R}^{n})} \leq C 
\displaystyle{\frac{|Q|^{\frac{1}{p_0}}}{ \| \chi _{Q} \|_{L^{p(\cdot)}_{\omega}}}}$, where $C$ is a positive constant independent of the molecule $m$.
\end{remark}

Next, we give a molecular reconstruction theorem for $H^{p(\cdot)}_{\omega}(\mathbb{R}^{n})$. For them, we need to introduce the following discrete maximal: given $\phi \in \mathcal{S}(\mathbb{R}^{n})$ and $f \in \mathcal{S}'(\mathbb{R}^{n})$, we define
$$M^{dis}_{\phi}f(x) = \sup_{j \in \mathbb{Z}} \left| (\phi^{j} \ast f) (x) \right|,$$
where $\phi^{j}(x) = 2^{jn} \phi(2^{j}x).$ From \cite[Lemma 3.2 and Proof of Theorem 3.3]{Nakai}, it follows that for all 
$f \in  \mathcal{S}'(\mathbb{R}^{n})$ and all $0 < \theta < 1$
\begin{equation} \label{maximal discreta}
\mathcal{M}_N f(x) \leq C \left[ M \left(\left(M^{dis}_{\phi}f \right)^{\theta}\right)(x)\right]^{\frac{1}{\theta}}, \,\,\, for \,\, all \,\, x \in \mathbb{R}^{n}, 
\end{equation}
if $N$ is sufficiently large. This inequality gives the following result.

\begin{lemma} \label{Hpw norm discreta} Let $\phi \in \mathcal{S}(\mathbb{R}^{n})$ and $f \in \mathcal{S}'(\mathbb{R}^{n})$.
If $\omega \in \mathcal{W}_{p(\cdot)}$, then $\| f \|_{H^{p(\cdot)}_{\omega}} \leq C \| M^{dis}_{\phi}f \|_{L^{p(\cdot)}_{\omega}}$,
where $C$ is a positive constant which does not depend on $f$.
\end{lemma}

\begin{proof} From (\ref{maximal discreta}) above and Lemma \ref{potencia s}, we have
\[
\| f \|_{H^{p(\cdot)}_{\omega}} = \| \mathcal{M}_N f \|_{L^{p(\cdot)}_{\omega}} \lesssim 
\left\| \left[ M \left(\left(M^{dis}_{\phi}f \right)^{\theta}\right)(\cdot)\right]^{\frac{1}{\theta}} \right\|_{L^{p(\cdot)}_{\omega}}
= \left\| M \left(\left(M^{dis}_{\phi}f \right)^{\theta}\right) \right\|_{L^{p(\cdot)/\theta}_{\omega^{\theta}}}^{1/\theta}.
\] 
By taking $\frac{1}{\theta} > s_{\omega, p(\cdot)}$, by \cite[Theorem 3.1]{Ho1} and Lemma \ref{potencia s}, we get
\[
\| f \|_{H^{p(\cdot)}_{\omega}} \lesssim
\left\| M \left(\left(M^{dis}_{\phi}f \right)^{\theta}\right) \right\|_{L^{p(\cdot)/\theta}_{\omega^{\theta}}}^{1/\theta}
\lesssim \left\| \left(M^{dis}_{\phi}f \right)^{\theta} \right\|_{L^{p(\cdot)/\theta}_{\omega^{\theta}}}^{1/\theta}
= \left\| M^{dis}_{\phi}f \right\|_{L^{p(\cdot)}_{\omega}}.
\]
Consequently, we obtain the desired result.
\end{proof}

\begin{theorem} \label{molecular reconst}
Let $p(\cdot) : \mathbb{R}^{n} \to (0, \infty)$ be a measurable function with $0 < p_{-} \leq p_{+} < \infty$, and  
$\omega \in \mathcal{W}_{p(\cdot)}$. Suppose that $0 < \theta < 1$ satisfies $\frac{1}{\theta} \in \mathbb{S}_{\omega, \, p(\cdot)}$. 
Let $\{ \lambda_j \}_{j=1}^{\infty}$ be a sequence of scalars and let $\{ Q_j \}_{j=1}^{\infty}$ be a family of cubes such that
\[
\mathcal{A} \left( \{ \lambda_j \}_{j=1}^{\infty}, \{ Q_j \}_{j=1}^{\infty}, p(\cdot), \omega, \theta \right) :=
\left\| \sum_{j=1}^{\infty} \left( \frac{|\lambda_j |}{\| \chi_{Q_j} \|_{L^{p(\cdot)}_{\omega}}} \right)^{\theta} \chi_{Q_j} \right\|_{L^{p(\cdot)/\theta}_{\omega^{\theta}}}^{1/\theta} < \infty.
\]
If $\{ m_j \}_{j=1}^{\infty}$ is a sequence of $\omega - (p(\cdot), p_0, \lfloor n s_{\omega, p(\cdot)} - n \rfloor)$ molecules,
with $p_0 > \theta \left( \kappa^{1/\theta}_{\omega, p(\cdot)} \right)'$, such that $m_j$ is centered at $Q_j$ for every $j \in \mathbb{N}$ and the series
\[
g := \sum_{j=1}^{\infty} \lambda_j m_j
\]
converges in $\mathcal{S}'(\mathbb{R}^{n})$, then $g \in H^{p(\cdot)}_{\omega}(\mathbb{R}^{n})$ with
\begin{equation} \label{Hpw norm molecular}
\| g \|_{H^{p(\cdot)}_{\omega}} \lesssim
\mathcal{A} \left( \{ \lambda_j \}_{j=1}^{\infty}, \{ Q_j \}_{j=1}^{\infty}, p(\cdot), \omega, \theta \right), 
\end{equation}
where the implicit constant in (\ref{Hpw norm molecular}) does not depend on $g$, $\{ \lambda_j \}_{j=1}^{\infty}$, and
$\{ Q_j \}_{j=1}^{\infty}$.
\end{theorem}

\begin{proof}
Let $\phi \in C_{c}^{\infty}(\mathbb{R}^{n})$ such that $\chi_{B(0,1)} \leq \phi  \leq \chi_{B(0,2)}$, we set 
$\phi_{2^{k}}(x) = 2^{kn}\phi(2^{k}x)$ where $k \in \mathbb{Z}$.  Since 
$g = \sum_j \lambda_j m_j$ in $\mathcal{S}'(\mathbb{R}^{n})$, it follows that
\[
|(\phi_{2^{k}} \ast g)(x) | \leq \sum_j |\lambda_j| |(\phi_{2^{k}} \ast m_j)(x)|, \,\,\, \text{for all} \,\, x \in \mathbb{R}^{n}, \,\,
\text{and all} \,\, k \in \mathbb{Z}.
\]
We observe that the argument used in \cite[Proof of Theorem 5.2]{Nakai} to obtain the pointwise inequality $(5.2)$ works in this setting, but considering now the conditions $m_1)$, $m_2)$ and $m_3)$ in Definition \ref{def molecule}. Therefore we get
\begin{equation} \label{ineq discrete}
M_{\phi}^{dis}g(x) \lesssim \sum_{j} |\lambda_{j}| \chi_{2\sqrt{n}Q_j}(x) M(m_j)(x) +   
\sum_{j} |\lambda_{j}| \frac{ \left[ M(\chi_{Q_j})(x) \right]^{\frac{n + d_{p(\cdot)} +1}{n}}}{\| \chi_{Q_j} \|_{L^{p(\cdot)}_{\omega}} }, 
\,\,\,\, (x \in \mathbb{R}^{n}),
\end{equation}
where $M$ is the Hardy-Littlewood maximal operator and $d_{p(\cdot)} = \lfloor n s_{\omega, p(\cdot)} - n \rfloor$.

Then, from (\ref{ineq discrete}) above and Lemma \ref{Hpw norm discreta}, we have
\[
\| g \|_{H^{p(\cdot)}_\omega} \lesssim \left\| M_{\phi}^{dis}g  \right\|_{L^{p(\cdot)}_\omega}
\]
\[
\lesssim \left\| \sum_j |\lambda_j| \chi_{2\sqrt{n}Q_j}  \cdot M(m_j)  \right\|_{L^{p(\cdot)}_\omega} +
\left\| \sum_j |\lambda_{j}| \frac{\left[ M(\chi_{Q_j})(\cdot) \right]^{\frac{n + d_{p(\cdot)} +1}{n}}}{\|\chi_{Q_j}\|_{L^{p(\cdot)}_{\omega}}}  \right\|_{L^{p(\cdot)}_\omega} = J_1 + J_2.
\]

Now, we consider $J_1$. The boundedness of the Hardy-Littlewood maximal operator $M$ on $L^{p_0}$, Remark \ref{molecule Lp} and 
Lemma \ref{2Q} yield
\begin{equation} \label{Mmj}
\left\| [M(m_j)]^{\theta} \right\|_{L^{p_{0}/\theta}(2\sqrt{n}Q_{j})} = 
\left\| M(m_j) \right\|_{L^{p_{0}}(2\sqrt{n}Q_{j})}^{\theta} \lesssim
\left\| m_j \right\|_{L^{p_{0}}(\mathbb{R}^{n})}^{\theta} 
\end{equation}
\[
\lesssim 
\frac{| Q_j |^{\frac{\theta}{p_{0}}}}{\left\| \chi _{Q_j }\right\|_{L^{p(\cdot)/\theta}_{\omega^{\theta}}}}
\lesssim 
\frac{| 2\sqrt{n} Q_j |^{\frac{\theta}{p_{0}}}}{\left\| \chi _{ 2\sqrt{n}Q_j }\right\|_{L^{p(\cdot)/\theta}_{\omega^{\theta}}}},
\]
where $0 < \theta < 1$ and $\frac{1}{\theta} \in \mathbb{S}_{\omega, p(\cdot)}$, and 
$p_0 > \theta \left( \kappa^{1/\theta}_{\omega, p(\cdot)} \right)'$. By applying the $\theta$-inequality and \cite[Lemma 5.4]{Ho1} with 
$b_j = \left( \chi_{2\sqrt{n} Q_j} \cdot M(m_j) \right)^{\theta}$, (\ref{Mmj}) and 
$A_j = \left\| \chi_{2\sqrt{n} Q_{j}} \right\|_{L^{p(\cdot)/\theta}_{\omega^{\theta}}}^{-1}$, we get
\[
J_1 \lesssim \left\| \sum_{j} \left(|\lambda_j| \chi_{2\sqrt{n}Q_j}  \cdot M(m_j) \right)^{\theta} \right\|^{1/\theta}_{L^{p(\cdot)/\theta}_{\omega^{\theta}}} \lesssim \left\| \sum_{j} \left( \frac{|\lambda_j|}{\left\| \chi_{2\sqrt{n}Q_{j}} \right\|_{L^{p(\cdot)}_{\omega}}} 
\right)^{\theta} \chi_{2\sqrt{n} Q_j} \right\|^{1/\theta}_{L^{p(\cdot)/\theta}_{\omega^{\theta}}}.
\]
Being $\chi_{2\sqrt{n} Q_j} \leq M(\chi_{Q_j})^{2}$, by Lemma \ref{2Q} and \cite[Theorem 3.1]{Ho1}, we have
\[
J_1 \lesssim \left\| \left\{ \sum_{j} \left( \frac{|\lambda_j|^{\theta/2}}{\left\| \chi_{Q_{j}} \right\|^{\theta/2}_{L^{p(\cdot)}_{\omega}}}M(\chi_{Q_j})\right)^{2}  \right\}^{1/2} \right\|^{2/\theta}_{L^{2p(\cdot)/\theta}_{\omega^{\theta/2}}} \lesssim
\left\| \sum_{j} \left( \frac{|\lambda_j |}{\| \chi_{Q_j} \|_{L^{p(\cdot)}_{\omega}}} \right)^{\theta} \chi_{Q_j} 
\right\|_{L^{p(\cdot)/\theta}_{\omega^{\theta}}}^{1/\theta}.
\]
So,
\begin{equation} \label{J1}
J_1 \lesssim 
\mathcal{A} \left( \{ \lambda_j \}_{j=1}^{\infty}, \{ Q_j \}_{j=1}^{\infty}, p(\cdot), \omega, \theta \right) < \infty.
\end{equation}

To estimate $J_2$, we write $r = \displaystyle{\frac{n + d_{p(\cdot)} +1}{n}}$. Thus
\[
J_2 \lesssim 
\left\| \left( \sum_j |\lambda_{j}| \frac{\left[ M(\chi_{Q_j})(\cdot) \right]^{r}}{\|\chi_{Q_j}\|_{L^{p(\cdot)}_{\omega}}} \right)^{1/r}
\right\|_{L^{rp(\cdot)}_{\omega^{1/r}}}^{r}.
\]
Since 
\[
r = \frac{n + d_{p(\cdot)} +1}{n} = \frac{n + \lfloor n s_{\omega, p(\cdot)} - n \rfloor + 1}{n} > s_{\omega, p(\cdot)},
\] 
to apply again \cite[Theorem 3.1]{Ho1} followed by the $\theta$-inequality we obtain
\[
J_2 \lesssim 
\left\| \sum_j \left( \frac{|\lambda_{j}|}{\|\chi_{Q_j}\|_{L^{p(\cdot)}_{\omega}}} \chi_{Q_j} \right)^{1/r}
\right\|_{L^{rp(\cdot)}_{\omega^{1/r}}}^{r} = 
\left\| \sum_j  \frac{|\lambda_{j}|}{\|\chi_{Q_j}\|_{L^{p(\cdot)}_{\omega}}} \chi_{Q_j}
\right\|_{L^{p(\cdot)}_{\omega}} 
\]
\[
\lesssim
\left\| \sum_j  \left( \frac{|\lambda_{j}|}{\|\chi_{Q_j}\|_{L^{p(\cdot)}_{\omega}}} \right)^{\theta} \chi_{Q_j}
\right\|_{L^{p(\cdot)(\theta}_{\omega^{\theta}}}^{1/\theta}.
\]
Hence, 
\begin{equation} \label{J2}
J_2 \lesssim 
\mathcal{A} \left( \{ \lambda_j \}_{j=1}^{\infty}, \{ Q_j \}_{j=1}^{\infty}, p(\cdot), \omega, \theta \right) < \infty.
\end{equation}
Finally, (\ref{J1}) and (\ref{J2}) give (\ref{Hpw norm molecular}) and, with that, $g \in H^{p(\cdot)}_{\omega}(\mathbb{R}^{n})$.
\end{proof}

\section{Weighted variable estimates for singular integrals}

Let $\Omega \in C^{\infty}(S^{n-1})$ with $\int_{S^{n-1}} \Omega(u) d\sigma(u)=0$. We define the operator $T$ by
\begin{equation} \label{sing op}
Tf(x) = \lim_{\epsilon \rightarrow 0^{+}} \int_{|y| > \epsilon} \frac{\Omega(y/|y|)}{|y|^{n}} f(x-y) \, dy, \,\,\,\, x \in \mathbb{R}^{n}.
\end{equation}
It is well known that the operator $T$ is bounded on $L^{p_0}(\mathbb{R}^{n})$ for all $1 < p_0 < +\infty$ and of weak-type $(1,1)$ 
(see e.g. \cite{Elias}).

We start studying the behavior of the operator $T$ on atoms. Then, we prove the boundedness of $T$ on $H^{p(\cdot)}_{\omega}$.

\begin{proposition} \label{atoms into molec T}
Let $T$ be the operator given by (\ref{sing op}) and let $p(\cdot) : \mathbb{R}^{n} \to (0, \infty)$ be a measurable function with 
$0 < p_{-} \leq p_{+} < \infty$, and $\omega \in \mathcal{W}_{p(\cdot)}$. Then, for some universal constant $C > 0$, 
$C (Ta)(\cdot)$ is a $\omega - (p(\cdot), p_0, \lfloor n s_{\omega, p(\cdot)} - n \rfloor)$ molecule for each 
$\omega - (p(\cdot), p_0, n + 2 \lfloor n s_{\omega, p(\cdot)} - n \rfloor + 2)$ atom $a(\cdot)$.
\end{proposition}

\begin{proof} It is well known that $\widehat{Tf}(\xi) = m(\xi) \widehat{f}(\xi)$, where the multiplier $m$ is homogeneous of degree $0$ and is indefinitely diferentiable on $\mathbb{R}^{n} \setminus \{0\}$ (see e.g. \cite{Elias}). Moreover, for 
$K(y) = \frac{\Omega(y/|y|)}{|y|^{n}}$ we have
\begin{equation}
|\partial^{\alpha}_{y} K(y) |\leq C |y|^{-n-|\alpha|}, \,\,\,\,  \textit{for all } \,\, y \neq 0 \,\, \textit{and all multi-index} \,\, \alpha. \label{k estimate}
\end{equation}

Let $p_0 > 1$ and $d_{p(\cdot)} = \lfloor n s_{\omega, p(\cdot)} - n \rfloor$. Given a $w - (p(\cdot), p_0, n+2d_{p(\cdot)}+2)$ atom 
$a(\cdot)$ with support in the cube $Q=Q(z, r)$ (with side length $\ell(Q)=r$), we have that
\begin{equation}
\| Ta \|_{L^{p_0}(2\sqrt{n}Q)} \leq C \| a \|_{L^{p_0}} \leq C \frac{|Q|^{1/p_0}}{\| \chi_Q \|_{L^{p(\cdot)}_{\omega}}}, \label{Ta 0}
\end{equation}
since $T$ is bounded on $L^{p_0}(\mathbb{R}^{n})$. In view of the moment condition of $a(\cdot)$ we obtain
\[
Ta(x)=  \int_{Q} K(x-y) a(y) dy = \int_{Q} [K(x-y) - q_{n+2d+2}(x,y)] a(y) dy, \,\, x \notin 2\sqrt{n} Q
\]
where $y \to q_{n+2d_{p(\cdot)}+2}(x,y)$ is the degree $n+2d_{p(\cdot)}+2$ Taylor polynomial of the function $y \rightarrow K(x-y)$ expanded around $z$. From the estimate in (\ref{k estimate}), and the standard estimate of the remainder term of the Taylor expansion, there exists $\xi$ between  $y$ and $z$ such that
\begin{equation} \label{Ta}
| Ta(x) | \leq C \| a \|_{1} \frac{\ell(Q)^{n+2d_{p(\cdot)}+3}}{|x - \xi|^{2n+2d_{p(\cdot)}+3}}
\leq C\frac{ \ell(Q)^{2n+2d_{p(\cdot)}+3}}{\| \chi_Q \|_{L^{p(\cdot)}_{\omega}}}  |x- z|^{-2n-2d_{p(\cdot)}-3}, \,\,\, x \notin 2\sqrt{n} Q,
\end{equation}
this inequality and a simple computation allow us to obtain
\begin{equation}
|Ta(x)| \leq C \| \chi_Q \|_{L^{p(\cdot)}_{\omega}}^{-1} \left(  1 + \frac{|x - z|}{\ell(Q)} \right)^{-2n - 2d_{p(\cdot)} - 3}, \,\,\, 
\textit{for all} \,\, x \notin 2\sqrt{n} Q.  \label{Ta 2}
\end{equation}
From the estimate in (\ref{Ta}) we obtain that the function $x \rightarrow x^{\alpha} Ta(x)$ belongs to $L^{1}(\mathbb{R}^{n})$ 
for each $|\alpha| \leq d_{p(\cdot)}$, so
\[
|((-2\pi i x)^{\alpha}Ta) \,\, \widehat{} \,\, (\xi) | = |\partial^{\alpha}_{\xi} (m(\xi) \widehat{a}(\xi))| = \left| \sum_{\beta \leq \alpha} c_{\alpha, \beta} \, (\partial^{\alpha - \beta}_{\xi} m)(\xi) \,  (\partial^{\beta}_{\xi} \widehat{a}) (\xi) \right|
\]
\[
= \left| \sum_{\beta \leq \alpha} c_{\alpha, \beta} \, (\partial^{\alpha - \beta}_{\xi} m)(\xi) \,  ((-2\pi i x)^{\beta} a) \,\, \widehat{} \,\, (\xi) \right|,
\]
from the homogeneity of the function $\partial^{\alpha - \beta}_{\xi} m$ we obtain that
\begin{equation}
|((-2\pi i x)^{\alpha}Ta) \,\, \widehat{} \,\, (\xi) |\leq C \sum_{\beta \leq \alpha}| c_{\alpha, \beta}|  \frac{\left| ((-2\pi i x)^{\beta} a) \,\, \widehat{} \,\, (\xi) \right|}{|\xi|^{|\alpha| -| \beta|}}, \,\,\,  \xi \neq 0. \label{limite}
\end{equation}
Since $\displaystyle{\lim_{\xi \rightarrow 0}}  \frac{\left| ((-2\pi i x)^{\beta} a) \,\, \widehat{} \,\, (\xi) \right|}{|\xi|^{|\alpha| -| \beta|}} =0$ for each $\beta \leq \alpha$ (see 5.4, pp. 128, in \cite{Stein}), taking the limit as $\xi \rightarrow 0$ at (\ref{limite}), we get 
\begin{equation}
\int_{\mathbb{R}^{n}} (-2\pi i x)^{\alpha} Ta(x) \, dx = ((-2\pi i x)^{\alpha}Ta) \,\, \widehat{} \,\, (0) =0, \,\,\, \textit{for all}
\,\,\, |\alpha| \leq d_{p(\cdot)}. \label{Ta 3} 
\end{equation}
From (\ref{Ta 0}), (\ref{Ta 2}) and (\ref{Ta 3}) it follows that there exists an universal constant $C>0$ such that $C(Ta)(\cdot)$ is 
a $w-(p(\cdot), p_0, d_{p(\cdot)})$ molecule if $a(\cdot)$ is a $w-(p(\cdot), p_0, n+2d_{p(\cdot)}+2)$ atom.
\end{proof}

\begin{theorem} \label{Boundedness T}
Let $p(\cdot) \in \mathcal{P}^{\log}(\mathbb{R}^{n})$, $\omega \in \mathcal{W}_{p(\cdot)}$ and let $T$ be the singular integral operator given by (\ref{sing op}). Then $T$ can be extended to a bounded operator on $H^{p(\cdot)}_{\omega}(\mathbb{R}^{n})$.
\end{theorem}

\begin{proof}
Given $\omega \in \mathcal{W}_{p(\cdot)}$, by Definition \ref{pesos Wp}, there exists $0 < \theta < 1$ such that 
$\frac{1}{\theta} \in \mathbb{S}_{\omega, \, p(\cdot)}$. Let $p_0 > \max \left\{ \theta \left( \kappa_{\omega, \, p(\cdot)}^{1/\theta} \right)', 1 \right\}$. Given $f \in \mathcal{S}_{0}(\mathbb{R}^{n})$, by Theorem \ref{w atomic decomp}, there exist a sequence of real numbers 
$\{\lambda_j\}_{j=1}^{\infty}$, a sequence of cubes $\{ Q_j \}_{j=1}^{\infty}$, and 
$\omega - (p(\cdot), p_0, n + 2 d_{p(\cdot)} + 2)$ atoms $a_j$ supported 
on $Q_j$, such that $f = \sum_{j=1}^{\infty} \lambda_j a_j$ converges in $L^{p_0}(\mathbb{R}^{n})$ and
\begin{equation} \label{norm A}
\mathcal{A}(\{ \lambda_j \}_{j=1}^{\infty}, \{Q_j\}_{j=1}^{\infty}, p(\cdot), \omega, \theta) \lesssim 
\| f \|_{H^{p(\cdot)}_{\omega}} < \infty.
\end{equation}
Since $f = \sum_{j=1}^{\infty} \lambda_j a_j$ converges in $L^{p_0}(\mathbb{R}^{n})$ and $T$ is bounded on $L^{p_0}$, we have 
that 
\begin{equation} \label{convergencia de T}
Tf = \sum_{j=1}^{\infty} \lambda_j Ta_j \,\,\, \text{in} \,\, \mathcal{S}'(\mathbb{R}^{n}).
\end{equation} 
Now, by (\ref{convergencia de T}) and (\ref{norm A}) above and Proposition \ref{atoms into molec T}, we can apply 
Theorem \ref{molecular reconst} to obtain
\[
\| Tf \|_{H^{p(\cdot)}_{\omega}}  \lesssim 
\mathcal{A}(\{ \lambda_j \}_{j=1}^{\infty}, \{Q_j\}_{j=1}^{\infty}, p(\cdot), \omega, \theta) 
\lesssim  \| f \|_{H^{p(\cdot)}_{\omega}},
\]
for all $f \in \mathcal{S}_{0}(\mathbb{R}^{n})$. Finally, the theorem follows from the density of
$\mathcal{S}_{0}(\mathbb{R}^{n})$ in $H^{p(\cdot)}_{\omega}(\mathbb{R}^{n})$ (see Proposition \ref{dense}).
\end{proof}

In particular, if $p(\cdot) \in \mathcal{P}^{\log}$ and $\omega \in \mathcal{W}_{p(\cdot)}$, then Hilbert transform and Riesz transforms admit a continuous extension on $H^{p(\cdot)}_{\omega}(\mathbb{R})$ and on $H^{p(\cdot)}_{\omega}(\mathbb{R}^{n})$ respectively.

\section{Weighted variable estimates for Riesz potential}

Let $0 < \alpha < n$. The \textit{Riesz potential} of order $\alpha$ is the operator $I_{\alpha }$ defined, 
say on $\mathcal{S}(\mathbb{R}^{n})$, by
\begin{equation}
I_{\alpha }f(x)=\int_{\mathbb{R}^{n}} f(y) |x-y|^{\alpha - n}dy, \,\,\, x \in \mathbb{R}^{n}.
\label{Ia}
\end{equation}
In \cite{Rocha4}, the author proved that the operator $I_{\alpha}$ extends to a bounded operator 
$H^{p(\cdot)}_{\omega}(\mathbb{R}^{n}) \to L^{q(\cdot)}_{\omega}(\mathbb{R}^{n})$, for $\frac{1}{p(\cdot)} := \frac{1}{q(\cdot)} + \frac{\alpha}{n}$; under the following assumptions:

$1)$ $q(\cdot) \in \mathcal{P}^{\log}(\mathbb{R}^{n})$ and $\omega \in \mathcal{W}_{q(\cdot)}$;

$2)$ for every cube $Q \subset \mathbb{R}^{n}$, 
$\| \chi_Q \|_{L^{q(\cdot)}_{\omega}} \approx |Q|^{-\alpha/n} \| \chi_Q \|_{L^{p(\cdot)}_{\omega}}$. \\
We observe that if $q(\cdot) \in \mathcal{P}^{\log}(\mathbb{R}^{n})$ and $\omega \equiv 1$, then the condition $2)$ holds. This was proved in \cite{Rocha1}. In \cite{Rocha4}, the author gave non trivial examples of power weights satisfying $2)$. So, the condition $2)$ is an admissible hypothesis.

Next, assuming the conditions $1)$ and $2)$ above, we will prove that Riesz potential $I_{\alpha}$ extends to a bounded operator $H^{p(\cdot)}_{\omega}(\mathbb{R}^{n}) \to H^{q(\cdot)}_{\omega}(\mathbb{R}^{n})$. For them, we will first show that $I_{\alpha}$ maps atoms into molecules.

\begin{proposition} \label{atoms into molecules} 
Let $0 < \alpha < n$, $q(\cdot) : \mathbb{R}^{n} \to (0, \infty)$ be a measurable function with 
$0 < q_{-} \leq q_{+} < \infty$, and $\omega \in \mathcal{W}_{q(\cdot)}$. If $\frac{1}{p(\cdot)}:= \frac{1}{q(\cdot)} + \frac{\alpha}{n}$ 
and $\| \chi_Q \|_{L^{q(\cdot)}_{\omega}} \approx |Q|^{-\alpha/n} \| \chi_Q \|_{L^{p(\cdot)}_{\omega}}$ for every cube $Q$, then, for some
universal constant $C>0$, $C(I_{\alpha} a)(\cdot)$ is a $\omega - (q(\cdot), q_0, \lfloor n s_{\omega, q(\cdot)} - n \rfloor)$ molecule for each $\omega - (p(\cdot), p_0, 2 \lfloor n s_{\omega, q(\cdot)} - n \rfloor + \lfloor \alpha \rfloor + 3 + n)$ atom $a(\cdot)$, where 
$q_0 > \frac{n}{n-\alpha}$ and $\frac{1}{p_0} := \frac{1}{q_0} + \frac{\alpha}{n}$.
\end{proposition}

\begin{proof} By \cite[Proposition 3.1]{Rocha4}, we have that $\mathcal{W}_{q(\cdot)} \subset \mathcal{W}_{p(\cdot)}$ and $s_{\omega, \, p(\cdot)} \leq s_{\omega, \, q(\cdot)} + \frac{\alpha}{n}$. Let $d_{q(\cdot)} := \lfloor n s_{\omega, q(\cdot)} - n \rfloor$ and $N = 2d_{q(\cdot)} + \lfloor \alpha \rfloor + 3 + n$, we observe that $N \geq d_{p(\cdot)}$. 
Fix $q_0 > \frac{n}{n-\alpha}$ and we put $\frac{1}{p_0} := \frac{1}{q_0} + \frac{\alpha}{n}$. 
Given a $\omega - (p(\cdot), p_0, N)$ atom $a(\cdot)$ supported on a cube $Q=Q(z,r)$, we will show that there exists an universal 
constant $C>0$ such that $C I_{\alpha}a(\cdot) $ is a $\omega - ( q(.), q_{0}, d_{q(\cdot)} )$ molecule centered at $Q$.
Indeed,

\vspace{0.2cm}
$m_1)$ by Sobolev's theorem, the condition $a_1)$ of the atom $a(\cdot)$, and since 
$\| \chi_Q \|_{L^{q(\cdot)}_{\omega}} \approx |Q|^{-\alpha/n} \| \chi_Q \|_{L^{p(\cdot)}_{\omega}}$ for every cube $Q$, we have
\[
\| I_{\alpha}a \|_{L^{q_{0}}(2\sqrt{n}Q)} \lesssim \| a \|_{L^{p_{0}}} \leq \frac{| Q |^{\frac{1}{p_{0}}}}
{\| \chi _{Q} \|_{L^{p(\cdot)}_{\omega}}} \lesssim \frac{| Q |^{\frac{1}{q_{0}}}}{\| \chi_{Q} \|_{L^{q(\cdot)}_{\omega}}};
\]

$m_2)$ by doing use of the conditions $a_3)$ and $a_2)$ of the atom $a(\cdot)$, as in the estimate (24) obtained 
in \cite[Theorem 5.1]{Rocha4}, we get, for $x$ outside $2\sqrt{n}Q,$ that
\begin{eqnarray*}
|(I_{\alpha}a)(x)| & \lesssim & \frac{\ell(Q)^{N + 1}}{|x-z|^{n-\alpha + N + 1}} \|a\|_{L^{1}} \\
& \lesssim & \frac{1}{\| \chi _{Q} \|_{L^{q(\cdot)}_{\omega}}} \left( \frac{\ell(Q)}{|x-z|}\right)^{n - \alpha + N + 1} \\
& \lesssim & \frac{1}{\| \chi _{Q} \|_{L^{q(\cdot)}_{\omega}}}\left( \frac{\ell(Q)}{\ell(Q)+ |x-z|}\right)^{2n+2d_{q(\cdot)}+3} \\
& = & \frac{1}{ \| \chi _{Q} \|_{L^{q(\cdot)}_{\omega}}} \left( 1+\frac{|x-z|}{\ell(Q)}\right)^{-2n-2d_{q(\cdot)}-3},
\end{eqnarray*}
where the third inequality follows since $x$ outside $2\sqrt{n}Q$ implies $\ell(Q)+ |x-z| < 2 |x-z|$;

$m_3)$ the moment condition was proved by Taibleson and Weiss in \cite{Taibleson}.
\\
Thus $I_{\alpha}a$ satisfies Definition {\ref{def molecule}} with an universal implicit constant. This completes the proof.
\end{proof}

\begin{theorem} \label{Boundedness Ialpha}
Let $0 < \alpha < n$, $q(\cdot) \in \mathcal{P}^{\log}(\mathbb{R}^{n})$ with $0 < q_{-} \leq q_{+} < \infty$ and 
$\omega \in \mathcal{W}_{q(\cdot)}$. If $\frac{1}{p(\cdot)} := \frac{1}{q(\cdot)} + \frac{\alpha}{n}$ and 
$\| \chi_Q \|_{L^{q(\cdot)}_{\omega}} \approx |Q|^{-\alpha/n} \| \chi_Q \|_{L^{p(\cdot)}_{\omega}}$ for every cube $Q$, then 
the Riesz potential $I_{\alpha}$ given by (\ref{Ia}) can be extended to a bounded operator 
$H^{p(\cdot)}_{\omega}(\mathbb{R}^{n}) \to H^{q(\cdot)}_{\omega}(\mathbb{R}^{n})$.
\end{theorem}

\begin{proof} 
Let $\omega \in \mathcal{W}_{q(\cdot)}$, by Definition \ref{pesos Wp}, there exists $0 < \theta < 1$ such that 
$\frac{1}{\theta} \in \mathbb{S}_{\omega, \, q(\cdot)}$. 
Now, we take $q_0 > \max \left\{ \theta \left( \kappa_{\omega, \, q(\cdot)}^{1/\theta} \right)', \frac{n}{n-\alpha} \right\}$, 
and define  $\frac{1}{p_0} := \frac{1}{q_0} + \frac{\alpha}{n}$. By \cite[Proposition 3.1]{Rocha4}, we have that $\mathcal{W}_{q(\cdot)} \subset \mathcal{W}_{p(\cdot)}$ and $s_{\omega, \, p(\cdot)} \leq s_{\omega, \, q(\cdot)} + \frac{\alpha}{n}$. 
So, given $f \in S_{0}(\mathbb{R}^{n})$, by Theorem \ref{w atomic decomp} there exist a sequence of real numbers 
$\{\lambda_j\}_{j=1}^{\infty}$, a sequence of cubes $\{ Q_j \}_{j=1}^{\infty}$, and 
$\omega - (p(\cdot), p_0, 2 \lfloor n s_{\omega, \, q(\cdot)} - n \rfloor + \lfloor \alpha \rfloor + 3 + n)$ atoms $a_j$ supported 
on $Q_j$, satisfying
\begin{equation} \label{norm Hpw2}
\mathcal{A}(\{ \lambda_j \}_{j=1}^{\infty}, \{Q_j\}_{j=1}^{\infty}, p(\cdot), \omega, \theta) \lesssim 
\| f \|_{H^{p(\cdot)}_{\omega}} < \infty,
\end{equation}
and $f = \sum_{j=1}^{\infty} \lambda_j a_j$ converges in $L^{p_0}(\mathbb{R}^{n})$. 
Then, by Sobolev's Theorem, $I_{\alpha}f = \sum_{j=1}^{\infty} \lambda_{j} I_\alpha(a_{j})$ in 
$L^{\frac{n p_{0}}{n- \alpha p_0}}(\mathbb{R}^{n})$ thus 
\begin{equation} \label{converg Riesz pot}
I_{\alpha}f = \sum_{j=1}^{\infty} \lambda_{j} I_\alpha(a_{j}) \,\,\, \text{in} \,\, \mathcal{S}'(\mathbb{R}^{n}).
\end{equation}
Now, by (\ref{converg Riesz pot}), (\ref{norm Hpw2}) and Proposition \ref{Aq menor Ap}, and
Proposition \ref{atoms into molecules}, we can apply Theorem \ref{molecular reconst} to obtain
\begin{eqnarray*}
\| I_{\alpha}f \|_{H^{q(\cdot)}_{\omega}} & \lesssim &
\mathcal{A}(\{ \lambda_j \}_{j=1}^{\infty}, \{Q_j\}_{j=1}^{\infty}, q(\cdot), \omega, \theta) \\
& \lesssim & \mathcal{A}(\{ \lambda_j \}_{j=1}^{\infty}, \{Q_j\}_{j=1}^{\infty}, p(\cdot), \omega, \theta) \\
& \lesssim & \| f \|_{H^{p(\cdot)}_{\omega}},
\end{eqnarray*}
for all $f \in \mathcal{S}_{0}(\mathbb{R}^{n})$. Finally, since $p(\cdot) \in \mathcal{P}^{\log}(\mathbb{R}^{n})$, 
$0 < p_{-} \leq p_{+} < \infty$ and $\omega \in \mathcal{W}_{p(\cdot)}$, the theorem follows from the density of
$\mathcal{S}_{0}(\mathbb{R}^{n})$ in $H^{p(\cdot)}_{\omega}(\mathbb{R}^{n})$ (see Proposition \ref{dense}).
\end{proof}

\bigskip
\address{
Departamento de Matem\'atica \\
Universidad Nacional del Sur \\
Bah\'{\i}a Blanca, 8000 Buenos Aires \\
Argentina}
{pablo.rocha@uns.edu.ar}


\begin{thebibliography}{99}

\bibitem{Coifman} {\small R. COIFMAN}, A real characterization of $H^{p}$, Studia Math. 51 (1974), 269-274.

\bibitem{Coifman2} {\small R. COIFMAN and G. WEISS}, Extensions of Hardy spaces and their use in analysis, Bull. Amer. Math. Soc. 83 (1977), 569-645.

\bibitem{Uribe3} {\small D. CRUZ-URIBE and D. WANG}, Variable Hardy Spaces, Indiana Univ. Math. J. 63 (2) (2014), 447-493.

\bibitem{Fefferman} {\small C. FEFFERMAN and E. STEIN}, $H^{p}$ spaces of several variables, Acta Math. 129 (1972), 137-193.
 
\bibitem{Ho1} {\small K.-P. HO}, Atomic decompositions of weighted Hardy spaces with variable exponents, Tohoku Math. J. 69 (2) (2017), 
no. 3, 383-413.

\bibitem{Ho2} {\small K.-P. HO}, Sublinear operators on weighted Hardy spaces with variable exponents, Forum Math. 31 (3) (2019), 607-617.

\bibitem{Latter} {\small R. H. LATTER}, A characterization of $H^{p}(\mathbb{R}^{n})$ in terms of atoms, Studia Math. 62 (1978), 93-101.

\bibitem{Nakai} {\small E. NAKAI and Y. SAWANO}, Hardy spaces with variable exponents and generalized Campanato spaces, J. Funct. Anal. 262 
(2012), 3665-3748.

\bibitem{Rocha2} {\small P. ROCHA}, A note on Hardy spaces and bounded linear operators, Georgian Math. J. 25 (1) (2018), 73-76.

\bibitem{Rocha3} {\small P. ROCHA}, Boundedness of generalized Riesz potentials on the variable Hardy spaces, J. Aust. Math. Soc. 104 
(2018), 255-273.

\bibitem{Rocha4} {\small P. ROCHA}, Inequalities for weighted spaces with variable exponents, (2022), Preprint: arXiv:2211.12218

\bibitem{Rocha1} {\small P. ROCHA and M. URCIUOLO}, Fractional type integral operators on variable Hardy spaces, Acta Math. Hung. 143 
(2) (2014), 502-514.

\bibitem{Elias} {\small E. STEIN}, Singular Integrals and Differentiability Properties of Functions, Princeton Univ. Press, 
Princeton, NJ, 1970.

\bibitem{Stein} {\small E. STEIN}, Harmonic Analysis: Real-Variable Methods, Orthogonality, and Oscillatory Integrals, Princeton
University Press, 1993.

\bibitem{Taibleson} {\small M. H. TAIBLESON and G. WEISS}, The molecular characterization of certain Hardy spaces, Ast\'{e}risque 77 (1980), 67-149.

\end{thebibliography}
\end{document}